\documentclass{amsart}

\usepackage[latin1]{inputenc}
\usepackage{amsfonts}
\usepackage{amsmath}
\usepackage{amsthm}
\usepackage{amssymb}
\usepackage{latexsym}
\usepackage{enumerate}

\newtheorem{theorem}{Theorem}[section]
\newtheorem{lemma}[theorem]{Lemma}
\newtheorem{proposition}[theorem]{Proposition}

\newtheorem{definition}[theorem]{Definition}

\numberwithin{equation}{section}

\begin{document}

\newcommand{\cc}{\mathfrak{c}}
\newcommand{\N}{\mathbb{N}}
\newcommand{\Q}{\mathbb{Q}}
\newcommand{\R}{\mathbb{R}}

\newcommand{\PP}{\mathbb{P}}
\newcommand{\forces}{\Vdash}
\newcommand{\dom}{\text{dom}}
\newcommand{\osc}{\text{osc}}

\title[On isomorphic embeddings]
{On universal spaces for the class of Banach spaces 
whose dual balls are uniform Eberlein compacts}

\author{Christina Brech}
\thanks{The first author was partially supported by FAPESP (2010/12639-1) and Pr\'o-reitoria de Pesquisa USP (10.1.24497.1.2).} 
\address{Departamento de Matem\'atica, Instituto de Matem\'atica e Estat\'\i stica, Universidade de S\~ao Paulo,
Caixa Postal 66281, 05314-970, S\~ao Paulo, Brazil}
\email{christina.brech@gmail.com}

\author{Piotr Koszmider}
\thanks{The second author was partially supported by Polish Ministry of
Science and Higher Education research grant N N201
386234. } 
\address{Institute of Mathematics, Polish Academy of Sciences,
ul. \'Sniadeckich 8,  00-956 Warszawa, Poland}
\address{Institute of Mathematics, Technical University of \L\'od\'z,
ul.\ W\'olcza\'nska 215, 90-924 \L\'od\'z, Poland}

\email{\texttt{piotr.math@gmail.com}}

\subjclass[2010]{Primary 46B26; Secondary 03E35, 46B03}

\begin{abstract}
For $\kappa$ being the first uncountable cardinal $\omega_1$ or
$\kappa$ being the cardinality of the continuum $\mathfrak{c}$, we prove that it is consistent that there is no Banach space of density $\kappa$ 
in which it is possible to isomorphically embed every Banach space
of the same density which has a uniformly G\^ateaux differentiable renorming
or, equivalently, whose dual unit ball with the weak$^*$ topology is a subspace of
a Hilbert space (a uniform Eberlein compact space). This complements a consequence of 
results of M. Bell and of M. Fabian, G. Godefroy, V. Zizler that 
assuming the continuum hypothesis, there is a universal space for all Banach spaces 
of density $\kappa=\mathfrak{c}=\omega_1$
which have a uniformly G\^ateaux differentiable renorming.
Our result implies, in particular, that $\beta \N\setminus \N$ may not map continuously onto a compact subset of a Hilbert space with the weak topology of density $\kappa=\omega_1$ or $\kappa=\mathfrak{c}$ and that a $C(K)$ space for some uniform Eberlein compact space $K$ may not embed isomorphically into $\ell_\infty/c_0$.
\end{abstract}

\maketitle

\section{introduction}

A classical result of Banach and Mazur states that every separable Banach space can be isometrically embedded into the Banach space $C([0,1])$. In this paper we deal with the problem of embedding nonseparable Banach spaces of a given class into a single nonseparable space. We are interested in the two most important uncountable densities, $\omega_1$ and $\mathfrak{c}$. Since a continuous onto map $\phi: K\rightarrow L$ (for $K,L$ compact Hausdorff spaces) gives an isometric embedding $T(f)=f\circ \phi$ of the Banach space $C(L)$ into $C(K)$, we have the following strictly related notions of universality:

\begin{definition} Let $\mathcal X$ be a class of Banach spaces. We say that $X\in \mathcal X$ is  (isometrically) universal space for the class $\mathcal X$ if for every $Y\in \mathcal X$ there is an (isometric) isomorphism onto its range $T:Y\rightarrow X$. 

Let $\mathcal K$ be a class of compact Hausdorff spaces. We say that $K\in \mathcal K$ is a universal space for the class $\mathcal K$ if for every $L\in \mathcal K$ there is a continuous onto map $T:K\rightarrow L$. 
\end{definition}

The following proposition shows that universal compact Hausdorff spaces and universal Banach spaces are closely related:

\begin{proposition}\label{compactbanach} 
Suppose $\mathcal K$ is a class of compact Hausdorff spaces and $\mathcal X$ is a class of Banach spaces such that for each $K\in \mathcal K$, $C(K)\in \mathcal X$ and for each $X\in \mathcal X$, the dual
unit ball $B_{X^*}$ with the weak$^*$ topology is in $\mathcal K$. If $K$ is a universal compact space for $\mathcal K$, then $C(K)$ is an isometrically universal Banach space for $\mathcal X$.
\end{proposition}
\begin{proof} Let $X\in \mathcal X$. It is well known that $X$ isometrically embeds into $C(B_{X^*})$. By the hypothesis $B_{X^*}\in \mathcal K$ and so $K$ continuously maps onto $B_{X^*}$, which means that $C(B_{X^*})$ isometrically embeds into $C(K)$. Composing this isometric embedding with the isometric embedding of $X$ into $C(B_{X^*})$, we obtain the desired embedding of $X$ into $C(K)$.
\end{proof}

Let $\kappa$ be an infinite cardinal. Examples of pairs of classes $\mathcal K$ and $\mathcal X$ satisfying the hypothesis
of the above proposition are:
\begin{itemize}
\item the class of all compact spaces of weight $\leq \kappa$ and the class of all Banach spaces of density $\leq \kappa$; 
\item the class of all Eberlein compact spaces (compact subsets of Banach spaces with the weak topology) of weight $\leq \kappa$ and the class of all weakly compactly generated Banach spaces (WCG) of density $\leq \kappa$;
\item the class of all uniform Eberlein compact spaces (compact subsets of Hilbert spaces with the weak topology) 
of weight $\leq \kappa$ and the class of all Banach spaces which have a uniformly G\^ateaux differentiable renorming
(UG) of density $\leq \kappa$.
\end{itemize}

The nonexistence of a universal Banach space implies, by the above, the nonexistence
of a corresponding universal compact space, but the opposite direction is not immediate. There are some 
reasons for this. First, $C(K)$ may be universal without $K$ being universal, as in the case
of $K=[0,1]$ for the class of separable Banach spaces and metrizable compact spaces
(however $C(K)$ is isomorphic to $C(\Delta)$ where $\Delta$ is the Cantor set, which is
universal for metrizable compact spaces). Secondly,
 there are universal Banach spaces which are not isometrically universal.
For example, consider a strictly convex renorming ($||x+y||=2$ for $||x||=||y||=1$ implies $x=y$) of $C([0,1])$ (which exists by Theorem 9 of \cite{clarkson}) and notice that this space continues to be an isomorphically universal space but cannot isometrically include spaces whose norm is not strictly convex.

In the case of nonseparable Banach spaces we encounter mainly negative results
concerning universality or results showing that to obtain a universal space one has
to assume additional set-theoretic axioms. For example,
it was proved in \cite{argyrosben} that there are no universal
WCG Banach spaces of density $\omega_1$ or $\mathfrak{c}$ nor universal Eberlein compact
spaces of weight $\omega_1$ or $\mathfrak{c}$.

Assuming the continuum hypothesis, the compact space $\beta \N\setminus \N$ is a
universal compact space of weight $\mathfrak{c}$ and by the above proposition
$\ell_\infty/c_0\equiv C(\beta\N\setminus \N)$ is an isometrically universal Banach space for the class of 
spaces of density $\mathfrak{c}=\omega_1$. 
However, it was shown in \cite{universal} that it is consistent that there is no isomorphically universal
space of density $\mathfrak{c}$. K. Thompson and second author noted that a 
version of the proof from \cite{universal} gives the consistency of the nonexistence
of a universal Banach space of density $\omega_1$. 

In \cite{bell} M. Bell showed that assuming $\kappa^\omega=\kappa$ 
there is a universal uniform Eberlein compact space of weight $\kappa$. M. Fabian, G. Godefroy and V. Zizler showed that
the class of Banach spaces satisfying the hypothesis of Proposition \ref{compactbanach} corresponding to the class of uniform Eberlein compact spaces is the class of UG Banach spaces. So, these two results imply by Proposition \ref{compactbanach} that
there are universal UG Banach spaces of density $\mathfrak{c}=\omega_1$. 
M. Bell also showed in \cite{bell} that it is consistent that there is no universal Eberlein compact
of density $\omega_1$.
The main result of this paper is to complement the above result by its Banach space version.
Actually the full result is considerably stronger:

\begin{theorem} \label{maintheorem}
For $\kappa = \omega_1$ and for $\kappa = \mathfrak{c}$, it is consistent that  there
is no Banach space $X$ of density $\kappa$ such that every UG Banach space of density $\kappa$ embeds into $X$. In particular, there is no universal Banach space for the class of
UG Banach spaces of density $\kappa$.
\end{theorem}

As in the case of the class of separable spaces the notions of universality and isometrical universality do not coincide
for the class of UG spaces: a universal UG Banach space may not be an isometrically universal UG Banach space and this follows from the fact that WCG Banach spaces can be renormed to be strictly convex (see \cite{amirlinden}).

The above theorem implies in particular that, consistently, $\ell_\infty/c_0$ does not
contain an isomorphic copy of  some UG Banach spaces of density $\mathfrak{c}$. This is related 
to some recent results of M. Dzamonja and L. Soukup \cite{mirnalajos} as well as to
recent results  of  S. Todorcevic (\cite{stevo})
concerning the consistent existence of $C(K)$s for $K$ Corson compact
which do not embed into $\ell_\infty/c_0$. 
Also, the arguments like in the proof
of Proposition \ref{compactbanach} give that $\beta\N\setminus \N$ cannot be mapped onto
some compact subset of a Hilbert space with density $K$ with the weak topology.

The notation is fairly standard, the Banach spaces terminology follows
\cite{fabian} and the set-theoretic terminology follows \cite{kunen}.
Given a function $f$, we denote by $G_f$ the graph of $f$. $Fn_{<\omega}(n,D)$
stands for all finite partial functions from $\{0,...n-1\}$ into $D$ and
$Fn_{<\omega}(\omega,D)$ (resp. $Fn_{\leq\omega}(\omega,D)$) for all finite (resp. countable) partial functions from $\omega$ into $D$. If $X$ is a set, then $[X]^1$ denotes the collection of one-element subsets of $X$.
If $\mathcal A$ is a Boolean algebra and $a\in\mathcal A$, then
$[a]$ denotes the basic clopen set of the Stone space of $\mathcal A$ determined by $a$, that is
the set of all ultrafilters of $\mathcal A$ containing $a$.
A partial order has precaliber $\omega_1$ if and only if every uncountable subset of it includes
an uncountable subset where every finite subset has a lower bound.

\begin{definition}\label{c-algebra}
A Boolean algebra $\mathcal A$ is a (c)-algebra if it is generated by elements $\{A_{\xi,i}: \xi<\kappa, i<\omega\}$ of countably many pairwise disjoint antichains $\mathcal A_i=\{A_{\xi,i}: \xi<\kappa\}$ (collections of pairwise disjoint elements of $\mathcal A$) such that 
\begin{enumerate}[\emph{(c)}]
\item $A_{\xi_1, i_1}\vee...\vee A_{\xi_m, i_m}\not=1_{\mathcal A}$ for any distinct pairs $(i_1,\xi_1), ...,(i_m, \xi_m) \in \omega \times \kappa$.
\end{enumerate}
\end{definition}

We use the fact proved in \cite{bell} (Theorems 2.1 and 2.2) based on the results of \cite{imageseberlein} that the Stone space of (c)-algebra $\mathcal A$ is a uniform Eberlein compact space. The link between uniform Eberlein compact spaces and UG Banach spaces is based on the following:

\begin{theorem}[Theorem 2, \cite{godefroyetal}]
$\ $
\begin{enumerate} 
\item  A  Banach  space  X  admits  an  equivalent  uniformly  G\^ateaux 
differentiable  norm  if and  only  if  the  dual  unit  ball  $B_{X^*}$ equipped  with  the  weak$^* $
topology is  a  uniform Eberlein  compact. 
\item A  compact  space  K  is  a  uniform  Eberlein  compact  if and  only  if $C(K)$
admits  an  equivalent uniformly  G\^ateaux 
differentiable  norm,  if and only  if  there  is  a Hilbert  space  $\mathcal H$  and  a 
bounded  linear operator  from  $\mathcal H$ onto  a dense  set  in $C(K)$. 
\end{enumerate}
\end{theorem}

The main idea of the proofs is to construct a large family of generic $c$-algebras and prove
that given a UG Banach space $X$ one of the algebras is suficiently generic over $X$ to
prevent a possibility of an isomorphic embedding of the Banach space $C(K)$
where K is the Stone space of the algebra. The $c$-algebras are constructed by
partial order of approximations (see definitions \ref{definitionP} and \ref{definitionQ}) and the method of forcing is used to make conclusions about their consistent existence  (see \cite{kunen}).

\section{Density $\omega_1$}

This section is devoted to the proof of the main result (Theorem \ref{maintheorem}) for the case
$\kappa=\omega_1$. The following partial order will approximate the generic $c$-algebras of
cardinality $\omega_1$, one of which will induce a Banach space UG without an isomorphic embedding into a given  Banach space UG.

\begin{definition}\label{definitionP}
$\mathbb{P}$ is the forcing notion consisting of conditions $p=(n_p, D_p, F_p)$ where $n_p \in \omega$, $D_p \in [\omega_1]^{<\omega}$ and $F_p \subseteq Fn_{<\omega}(n_p,D_p)$ are such that $|F_p| < \omega$ and $[n_p \times D_p]^1 \subseteq F_p$,
ordered by $p \leq q$ if $n_p \geq n_q$, $D_p \supseteq D_q$, $F_p \supseteq F_q$ and 
\begin{enumerate}[\emph{(P)}]
\item given $f \in F_p$, there is $g \in F_q$ such that $G_f \cap (n_q \times D_q) \subseteq G_g$.
\end{enumerate}
\end{definition}

Given a model $V$ and a $\mathbb{P}$-generic filter $G$ over $V$, for each $\xi \in \omega_1$ and each $i \in \omega$, we define in $V[G]$ the following set:
$$A_{\xi,i} = \{f \in Fn_{<\omega}(\omega,\omega_1): \exists p \in G\text{ such that } f\in F_p \text{ and } f(i)=\xi\}.$$
Let $\mathcal{B}$ be the 
subalgebra of the Boolean algebra $\wp(Fn_{<\omega}(\omega,\omega_1))$
generated by the sets $\{A_{\xi,i}: (i,\xi) \in \omega \times \omega_1\}$.

\begin{lemma}\label{densityP}
For every $\xi \in \omega_1$, $i \in \omega$ and every $q\in \PP$ there is
$p\leq q$ such that $\xi\in D_p$ and $i<n_p$.
\end{lemma}
\begin{proof}
Let $n_p=\max(i+1,n_q)$ and $D_q=D_p\cup\{\xi\}$, $F_p=F_q\cup[n_p\times D_p]^1$.
Clearly $p\in\PP$. As the graph of every new partial function in $F_p$ has empty
intersection with $(n_q \times D_q)$, we conclude that $p\leq q$.
\end{proof}

\begin{proposition}\label{UEP}
In $V[G]$, we have that for every $i \in \omega$, $(A_{\xi,i})_{\xi<\omega_1}$ are pairwise disjoint nonempty sets such that whenever $(i_1,\xi_1),...,(i_k,\xi_k), (i, \xi)$ are distinct
pairs, then 
$$A_{i,\xi}\setminus (A_{i_1,\xi_1}\cup...\cup A_{i_k,\xi_k})\not=\emptyset.$$
Therefore, the Boolean algebra $\mathcal{B}$ generated by them is a c-algebra of cardinality $\omega_1$ and its Stone space $K$ is a uniform Eberlein compact space of weight $\omega_1$.
\end{proposition}
\begin{proof}
For all distinct $(i_1,\xi_1),...,(i_k,\xi_k), (i, \xi)$, by \ref{densityP}, we have 
$$\{(i,\xi)\}\in A_{i,\xi}\setminus (A_{i_1,\xi_1}\cup...\cup A_{i_k,\xi_k}).$$
So each $A_{\xi,i}$ is nonempty and the condition (c) of Definition \ref{c-algebra} is satisfied. It also follows directly from the fact that $f$'s are functions and that $(i, \xi) \in A_{\xi, i}$ that $(A_{\xi,i})_{\xi<\omega_2}$ are pairwise disjoint and nonempty.
\end{proof}

For each $(i, \xi) \in \omega \times \omega_1$, let $\dot{A}_{\xi,i}$ be a $\mathbb{P}$-name for $A_{\xi,i}$. 

\begin{definition}
Given $p_1 = (n_1, D_1, F_1), p_2 = (n_2, D_2, F_2) \in \mathbb{P}$, we say that they are isomorphic if $n_1=n_2$ and there is an order-preserving bijection $e: D_1 \rightarrow D_2$ such that $e|_{D_1 \cap D_2} = id$ and for all $f \in Fn_{<\omega}(\omega, \omega_1)$, $f \in F_1$ if and only if $e[f] \in F_2$, where $e[f](i) = e(f(i))$. 
\end{definition}

\begin{lemma}\label{minimalAmalgamation}
Let $p_k= (n, D_k, F_k)$ in $\mathbb{P}$, for $1 \leq k \leq m$, be pairwise isomorphic conditions such that $(D_k)_{1 \leq k\leq m}$ is a $\Delta$-system with root $D$. Then, there is $p \in\mathbb{P}$, $p \leq p_1, \dots, p_m$ such that for any distinct $1\leq i,i'\leq n$, any distinct $1\leq k,k'\leq m$ and any $\xi\in D_k\setminus D$ and $\xi'\in D_{k'}\setminus D$ we have 
$$p \Vdash \dot{A}_{\check{\xi},i} \cap \dot{A}_{\check{\xi}',i'}=  \emptyset.$$
\end{lemma}
\begin{proof}
We define $p= (n_p, D_p, F_p)$ by letting $n_p = n$, $D_p = D_1 \cup \dots \cup D_m$ and $F_p = F_1 \cup \dots \cup F_m$. Let us prove that $p \in \mathbb{P}$ and that $p \leq p_k$ for every $1 \leq k \leq m$. 

As finite unions of finite sets, $D_p$ and $F_p$ are countable. Notice that $[n \times D_p]^{1} = [n \times D_1]^1 \cup \dots \cup [n \times D_m]^1 \subseteq F_1 \cup \dots \cup F_m =F_p$ and that $F_p \subseteq Fn_{< \omega}(n, D_1) \cup \dots \cup Fn_{< \omega} (n, D_m) \subseteq Fn_{< \omega} (n, D_p)$, so that $p \in \mathbb{P}$.

Fix $1 \leq k \leq m$ and let us now show that $p \leq p_k$. By the definition of $p$, $n_p = n$, $D_p \supseteq D_k$ and $F_p \supseteq F_k$. Given $f \in F_p$, (P) of Definition \ref{definitionP} is trivially satisfied if $f \in F_k$. So, suppose $f \in F_{k'} \setminus F_k$ for some $k'\neq k$ and let $e$ be the order-preserving bijection from $D_{k'}$ onto $D_k$. Since $p_k$ and $p_{k'}$ are isomorphic and $f \in F_{k'}$, we get that $g = e[f] \in F_k$. Let us prove that $G_f \cap (n_k \times D_k) = G_g \cap (n_p \times D) \subseteq G_g$: given $(i,\xi) \in G_f \cap (n_k \times D_k)$, we have that $f(i)=\xi$ and since $f \in F_{k'} \subseteq n_p \times D_{k'}$, we get that $\xi \in D_k \cap D_{k'} = D$. Then $g(i) = e[f](i) = e(f(i)) = e(\xi) = \xi$, since $\xi \in D$. This proves that $G_f \cap (n_k \times D_k) \subseteq G_g$ and concludes the proof that $p \leq p_k$ for all $1 \leq k\leq m$.

Now let $1\leq i, i'\leq n$ be distinct, $1\leq k<k'\leq m$ and let $\xi\in D_k\setminus D$ and $\xi'\in D_{k'}\setminus D$. 
Note that in $F_p$ there is no function $f$ such that $f(i)=\xi$ and $f(i')=\xi'$. Let $G$ be a $\PP$-generic filter such that $p\in G$. Given any $p'\in G$, there is $p''\in G$ such that $p''\leq p,p'$. By (P) of Definition \ref{definitionP}, no function $g\in F_{p''}$ satisfies $g(i)=\xi$ and $g(i')=\xi'$, hence there is no such function in $p'$. It follows from the definition of $A_{\xi,i}$'s that $A_{\xi,i}\cap A_{\xi',i'}=\emptyset$ in $V[G]$. Since $G$ was an arbitrary $\mathbb{P}$-generic filter containing $p$, we conclude that $p \Vdash \dot{A}_{\check{\xi},i} \cap \dot{A}_{\check{\xi}',i'}=  \emptyset$.
\end{proof}

\begin{theorem}\label{cccP}
$\mathbb{P}$ has precaliber $\omega_1$ and hence satisfies the c.c.c.
\end{theorem}
\begin{proof}
Let $(p_\alpha)_{\alpha<\omega_1} \subseteq \mathbb{P}$, where each $p_\alpha = (n_\alpha, D_\alpha, F_\alpha)$. By the $\Delta$-system lemma, we may assume that $(D_\alpha)_{\alpha<\omega_1}$ is a $\Delta$-system with root $D$. By standard counting arguments, we may assume without loss of generality that $n_\alpha = n$ for every $\alpha<\omega_1$ and some fixed $n \in \omega$. By thinning out using further counting arguments, we can assume that for every $\alpha<\beta<\omega_1$, $p_\alpha$ and $p_\beta$ are isomorphic. Now, given $\alpha_1<... < \alpha_n< \omega_1$, by Lemma \ref{minimalAmalgamation} there is $p \leq p_{\alpha_1},... p_{\alpha_n}$.
\end{proof}

\begin{lemma}\label{strongAmalgamation}
Let $p_k= (n, D_k, F_k)$ in $\mathbb{P}$, for $1 \leq k \leq m$, be pairwise isomorphic conditions such that $(D_k)_{1 \leq k\leq m}$ is a $\Delta$-system with root $D$. Given, for all $1 \leq k \leq m$, $\xi_k \in D_k \setminus D$ and distinct $i_k < n$, there is $p \in\mathbb{P}$, $p \leq p_1, \dots, p_m$ such that
$$p \Vdash \dot{A}_{\check{\xi}_1,i_1} \cap \dots \cap \dot{A}_{\check{\xi}_m,i_m} \neq \emptyset.$$
\end{lemma}
\begin{proof}
We define $p= (n_p, D_p, F_p)$ by letting $n_p = n$, $D_p = D_1 \cup \dots \cup D_m$ and 
$$F_p = F_1 \cup \dots \cup F_m \cup \{f_0\},$$
where $f_0= \{(i_1,\xi_1), \dots, (i_m, \xi_m)\}$. Let us check that $p \in \mathbb{P}$ and that $p \leq p_k$ for every $1 \leq k \leq m$. 

As finite unions of finite sets, $D_p$ and $F_p$ are finite. Notice that $[n \times D_p]^1 = [n \times D_1]^1 \cup \dots \cup [n \times D_m]^1 \subseteq F_1 \cup \dots \cup F_m \subseteq F_p$ and that $F_p \setminus \{f_0\} \subseteq Fn_{< \omega}(n, D_1) \cup \dots \cup Fn_{< \omega} (n, D_m) \subseteq Fn_{< \omega} (n, D_p)$. Since $f_0 \in Fn_{< \omega}(n, D_1 \cup \dots \cup D_m) \subseteq Fn_{< \omega} (n, D_p)$, we get that $p \in \mathbb{P}$.

Fix $1 \leq k \leq m$ and let us now verify that $p \leq p_k$. By the definition of $p$, $n_p = n$, $D_p \supseteq D_k$ and $F_p \supseteq F_k$. To check (P) of Definition \ref{definitionP}, given $f \in F_p$, let us consider three cases:

\textbf{Case 1.} If $f \in F_k$, then (P) is trivially satisfied. 

\textbf{Case 2.} If $f \in F_{k'} \setminus F_k$ for some $k'\neq k$, let $e$ be the order-preserving bijection from $D_{k'}$ onto $D_k$. Since $p_k$ and $p_{k'}$ are isomorphic and $f \in F_{k'}$, we get that $g = e[f] \in F_k$. Let us prove that $G_f \cap (n_k \times D_k) = G_g \cap (n_p \times D) \subseteq G_g$: given $(i,\xi) \in G_f \cap (n_k \times D_k)$, we have that $f(i)=\xi$ and since $f \in F_{k'} \subseteq n_p \times D_{k'}$, we get that $\xi \in D_k \cap D_{k'} = D$. Then $g(i) = e[f](i) = e(f(i)) = e(\xi) = \xi$, since $\xi \in D$. This proves that $G_f \cap (n_k \times D_k) \subseteq G_g$, as we wanted.

\textbf{Case 3.} If $f = f_0 = \{(i_1,\xi_1), \dots, (i_m, \xi_m)\}$, then $G_f \cap (n_k \times D_k) = \{(i_k, \xi_k)\} \subseteq F_k$ by Definition \ref{definitionP}.

Finally, since $f_0 \in F_p$, $p$ forces that $\check{f}_0 \in \dot{A}_{\check{\xi}_1,i_1} \cap \dots \cap \dot{A}_{\check{\xi}_m,i_m}$, so that 
$$p \Vdash \dot{A}_{\check{\xi}_1,i_1} \cap \dots \cap \dot{A}_{\check{\xi}_m,i_m} \neq \emptyset,$$
which concludes the proof.
\end{proof}

\begin{definition}
$\Sigma$ is the product of $\omega_2$ copies of $\mathbb{P}$, with finite supports, that is,
$$\Sigma = \{\sigma: dom(\sigma) \rightarrow \mathbb{P}: dom(\sigma) \in [\omega_2]^{< \omega}\},$$
ordered by $\sigma_1 \leq \sigma_2$ if $dom(\sigma_1) \supseteq dom(\sigma_2)$ and for every $\alpha \in dom(\sigma_2)$, $\sigma_1(\alpha) \leq \sigma_2(\alpha)$. Given $A \subseteq \omega_2$, let 
$$\Sigma_A = \{\sigma \in \Sigma: dom(\sigma) \subseteq A\}.$$
\end{definition}

\begin{proposition}\label{isomorphismP}
Given $\gamma_0 \in \omega_2$, if $A = \omega_2 \setminus \{\gamma_0\}$, then the forcing $\Sigma_A$ is isomorphic to $\Sigma$.
\end{proposition}
\begin{proof}
Lift to $\PP$ a bijection between $\omega_2$ and $\omega_2\setminus\{\gamma_0\}$.
\end{proof}

\begin{theorem}
$V^{\Sigma} \vDash$ ``there is no Banach space $X$ of density $\omega_1<\mathfrak{c}$ such that for every uniform Eberlein compact space $K$ of weight at most $\omega_1$, $C(K)$ can be isomorphically embedded into $X$''.
\end{theorem}
\begin{proof}

We work in $V$. 
By contradiction, suppose that such a Banach space exists in $V^\Sigma$ and let $\dot{X}$ be a $\Sigma$-name for it. Since $\Sigma$ forces that $\dot{X}$ has density $\omega_1$, let $(\dot{v}_\eta)_{\eta <\omega_1}$ be a family of $\Sigma$-names such that $\Sigma \Vdash \{\dot{v}_\eta: \eta < \omega_1\}$ is a dense 
subset\footnote{Unfortunately, the Banach space $\dot X$ will not belong to any intermediate model, so
we cannot make use of any factorization of $\Sigma$ unless we are willing to deal with
normed non-complete spaces over the rationals and approximations (possibly nonlinear)
of linear operators into its completion. We have choosen a simpler approach, in our opinion, 
and work in $V$ with the full product $\Sigma$.} of $\dot{X}$. 

For each $F \in [\omega_1]^{<\omega}$, let $A_F \subseteq \Sigma$ be a maximal antichain in $\Sigma$ such that for every $\sigma \in A_F$, 
$$\text{either }\sigma \Vdash \Vert \sum_{\eta\in F}\dot{v}_{\eta} \Vert > 2 \quad \text{ or }\sigma \Vdash \Vert \sum_{\eta\in F}\dot{v}_{\eta} \Vert \leq 2.$$
Since $\Sigma$ is ccc, each $A_F$ is countable and for each $\sigma \in A_F$, $|dom(\sigma)| < \omega$. Then, the set
$$\bigcup \{dom(\sigma): \sigma \in A_F \text{ for some } F \in [\omega_1]^{<\omega} \}$$
has cardinality at most $\omega_1$. Fix $\gamma_0 \in \omega_2$ such that for every $F \in [\omega_1]^{<\omega}$ and every $\sigma \in A_F$, $\gamma_0 \notin dom(\sigma)$.

Since $\Sigma \sim \Sigma_A \times \Sigma_{\{\gamma_0\}} \sim \Sigma_A \times \mathbb{P}$, let us consider in $V^{\Sigma_A \times \Sigma_{\{\gamma_0\}}}$, the Stone space $K_{\gamma_0}$ of the algebra generated by the family $(A_{\xi,i}(\gamma_0))_{\xi\in \omega_1, i \in \omega}$ added by $\Sigma_{\{\gamma_0\}} \sim \mathbb{P}$ over $V^{\Sigma_A}$. By Proposition \ref{UEP}, $K_{\gamma_0}$ is a uniform Eberlein compact of weight $\omega_1$ and let us show that $C(K_{\gamma_0})$ does not embed isomorphically into $\dot{X}$. For each $(i, \xi) \in \omega \times \omega_1$, let $\dot{A}_{\xi, i}$ be a $\Sigma$-name for $A_{\xi,i}(\gamma_0)$.

Suppose there is $\dot{T}: C(K_{\gamma_0}) \rightarrow \dot{X}$ an isomorphic embedding and without loss of generality, assume that $\Sigma \Vdash \Vert \dot{T} \Vert = 1$. Let $p_0 \in \Sigma$, now find
$p_1\leq p_0$ and  $m \in \omega$ such that $p_1 \Vdash 3\cdot \Vert \dot{T}^{-1} \Vert <\check m$.

For each $\alpha< \omega_1$, let $\{\xi_k(\alpha): 1 \leq k \leq m\} \subseteq \omega_1 \setminus \alpha$ 
have cardinality $m$ and take $\sigma_\alpha \in \Sigma$ 
with $\sigma_\alpha\leq p_1$ and $\eta_1(\alpha), \dots \eta_m(\alpha) \in \omega_1$ such that 
$$\sigma_\alpha \Vdash \forall 1 \leq k \leq m \quad \Vert \dot{T}(\chi_{\dot{A}_{\check{\xi}_k(\alpha), k}(\gamma_0)}) - \dot{v}_{\check{\eta}_k(\alpha)}\Vert < \frac{1}{m}.$$
Let $s_\alpha = \sigma_\alpha|_{\omega_2 \setminus \{\gamma_0\}} \in \Sigma_{\omega_2 \setminus \{\gamma_0\}}$ and $\sigma_\alpha(\gamma_0) = (n_\alpha, D_\alpha, F_\alpha)$. Without loss of generality, we may assume that $\xi_1(\alpha), \dots , \xi_m(\alpha) \in D_\alpha$. Using the $\Delta$-system lemma, we may assume that $(D_\alpha)_{\alpha<\omega_1}$ is a $\Delta$-system with root $D$ and that $\xi_1(\alpha), \dots , \xi_m(\alpha) \in D_\alpha \setminus D$ for each $\alpha<\omega_1$
and for all $\alpha<\omega_1$ the conditions $\sigma_\alpha(\gamma_0)$ are pairwise
isomorphic.

By Proposition \ref{isomorphismP}, $\Sigma_{\omega_2 \setminus \{\gamma_0\}}$ is isomorphic to $\Sigma$ and since, by the productivity of precaliber $\omega_1$ and
 Lemma \ref{cccP} $\Sigma$ has precaliber $\omega_1$, so does $\Sigma_{\omega_2 \setminus \{\gamma_0\}}$. So, given $\alpha_1< \dots < \alpha_m$, there is $s \in \Sigma_{\omega_2 \setminus \{\gamma_0\}}$ such that $s \leq s_{\alpha_1}, \dots, s_{\alpha_m}$.

Let $F = \{ \eta_1(\alpha_1), \dots, \eta_m(\alpha_m) \}$ and since $A_F$ is a maximal antichain in $\Sigma$ which is contained in $\Sigma_{\omega_2 \setminus \{\gamma_0\}}$, $A_F$ is also a maximal antichain in $\Sigma_{\omega_2 \setminus \{\gamma_0\}}$. Then, there is $s' \in A_F$ and $s'' \in \Sigma_{\omega_2 \setminus \{\gamma_0\}}$ such that $s'' \leq s, s'$. 

Since $s''\leq s'$ and $s'\in A_F$, either
$$s'' \Vdash \Vert \dot{v}_{\check{\eta}_1(\alpha_1)} + \dots + \dot{v}_{\check{\eta}_m(\alpha_m)} \Vert \leq 2$$
or 
$$s'' \Vdash \Vert \dot{v}_{\check{\eta}_1(\alpha_1)} + \dots + \dot{v}_{\check{\eta}_m(\alpha_m)} \Vert > 2.$$
Let us now get a contradiction.

\textbf{Case 1.} $s''$ forces that $\Vert \dot{v}_{\check{\eta}_1(\alpha_1)} + \dots + \dot{v}_{\check{\eta}_m(\alpha_m)} \Vert > 2$.

In this case, let $p \in \mathbb{P}$ be such that $p \leq \sigma_{\alpha_1}(\gamma_0), \dots, \sigma_{\alpha_m}(\gamma_0)$ obtained by Lemma \ref{minimalAmalgamation} and let $\sigma = (s'', p) \in \Sigma$. Then, $\sigma \leq s'', p$ and $\sigma \leq \sigma_{\alpha_1}, \dots, \sigma_{\alpha_m}$. So,
$$\sigma \Vdash \forall 1 \leq k < k'\leq m \quad \dot{A}_{\check{\xi}_k(\alpha_k),i_k}(\gamma_0)\cap \dot{A}_{\check{\xi}_{k'}(\alpha_{k'}),i_{k'}}(\gamma_0) = \emptyset$$
so that 
$$\sigma \Vdash \Vert \chi_{[\dot{A}_{\check{\xi}_1(\alpha_1),i_1}(\gamma_0)]} + \dots + \chi_{[\dot{A}_{\check{\xi}_m(\alpha_m),i_m}(\gamma_0)]} \Vert = 1.$$
Since $\sigma \leq \sigma_{\alpha_1}, \dots, \sigma_{\alpha_m}$, 
$$\sigma \Vdash \forall 1 \leq k \leq m \quad \Vert \dot{T}(\chi_{[\dot{A}_{\check{\xi}_k(\alpha_k),i_k}(\gamma_0)]}) - \dot{v}_{\check{\eta}_k(\alpha_k)} \Vert < \frac{1}{m},$$
so that, using the fact that $\Sigma \Vdash \Vert \dot{T} \Vert =1$, we get that 
$$\begin{array}{ll}
\sigma \Vdash & \Vert \dot{v}_{\check{\eta}_1(\alpha_1)} + \dots + \dot{v}_{\check{\eta}_m(\alpha_m)} \Vert\\
& \leq \Vert \displaystyle\sum_{k=1}^{m}  (\dot{v}_{\check{\eta}_k(\alpha_k)}-\dot{T}(\chi_{[\dot{A}_{\check{\xi}_k(\alpha_k),i_k}(\gamma_0)]})) +\displaystyle\sum_{k=1}^{m} \dot{T}(\chi_{[\dot{A}_{\check{\xi}_k(\alpha_k),i_k}(\gamma_0)]})\Vert\\
& \leq\Vert \chi_{[\dot{A}_{\check{\xi}_1(\alpha_1),i_1}(\gamma_0)]} + \dots + \chi_{[\dot{A}_{\check{\xi}_m(\alpha_m),i_m}(\gamma_0)]} \Vert + m \cdot \frac{1}{m} \leq 2,
\end{array}$$
contradicting the fact that $\sigma \leq s''$ and 
$$s''\Vdash \Vert \dot{v}_{\check{\eta}_1(\alpha_1)} + \dots + \dot{v}_{\check{\eta}_m(\alpha_m)} \Vert > 2.$$

\textbf{Case 2.} $s''$ forces that $\Vert \dot{v}_{\check{\eta}_1(\alpha_1)} + \dots + \dot{v}_{\check{\eta}_m(\alpha_m)} \Vert \leq 2$.

In this case, let $p \in \mathbb{P}$ be such that $p \leq \sigma_{\alpha_1}(\gamma_0), \dots, \sigma_{\alpha_m}(\gamma_0)$ obtained by Lemma \ref{strongAmalgamation} and let $\sigma = (s'', p) \in \Sigma$. Then, $\sigma \leq s'', p$ and $\sigma \leq \sigma_{\alpha_1}, \dots, \sigma_{\alpha_m}$. So,
$$\sigma \Vdash \dot{A}_{\check{\xi}_1(\alpha_1),i_1}(\gamma_0) \cap \dots \cap \dot{A}_{\check{\xi}_{m}(\alpha_{m}),i_{m}}(\gamma_0) \neq \emptyset$$
so that 
$$\sigma \Vdash \Vert \chi_{[\dot{A}_{\check{\xi}_1(\alpha_1),i_1}(\gamma_0)]} + \dots + \chi_{[\dot{A}_{\check{\xi}_m(\alpha_m),i_m}(\gamma_0)]} \Vert = m.$$
Since $\sigma \leq \sigma_{\alpha_1}, \dots, \sigma_{\alpha_m}$, 
$$\sigma \Vdash \forall 1 \leq k \leq m \quad \Vert \dot{T}(\chi_{[\dot{A}_{\check{\xi}_k(\alpha_k),i_k}(\gamma_0)]})- \dot{v}_{\check{\eta}_k(\alpha_k)} \Vert < \frac{1}{m},$$
so that, using the fact that $\sigma \leq s''$, we get that
$$\begin{array}{ll}
\sigma\Vdash & \Vert \dot{T} (\chi_{[\dot{A}_{\check{\xi}_1(\alpha_1),i_1}]} + \dots + \chi_{[\dot{A}_{\check{\xi}_m(\alpha_m),i_m}]}) \Vert\\
& \leq \Vert \displaystyle\sum_{k=1}^{m}  (\dot{T}(\chi_{[\dot{A}_{\check{\xi}_k(\alpha_k),i_k}(\gamma_0)]})-\dot{v}_{\check{\eta}_k(\alpha_k)}) +\displaystyle\sum_{k=1}^{m} \dot{v}_{\check{\eta}_k(\alpha_k)}\Vert\\
& \leq \Vert \dot{v}_{\check{\eta}_1(\alpha_1)} + \dots + \dot{v}_{\check{\eta}_m(\alpha_m)}\Vert + m \cdot \frac{1}{m} \leq 3,\end{array}$$
which contradicts the fact that $\Sigma\Vdash 3 \cdot \Vert \dot{T}^{-1} \Vert < m$.

Since the condition $p_0\in\Sigma$ was arbitrary, we showed that a dense subset of $\Sigma$ forces
the nonexistence of the embedding $\dot T$, hence it does not exist in $V[G]$.
\end{proof}

\section{Density $\mathfrak{c}$}

This section is devoted to the proof of the main result  \ref{maintheorem} for the case
$\kappa=\mathfrak{c}$. The following partial order will approximate the generic $c$-algebras of
cardinality $\omega_2=\mathfrak{c}$, one of which will induce a Banach space UG without an isomorphic embedding into a given  Banach space UG.

\begin{definition}\label{definitionQ}
$\mathbb{Q}$ is the forcing notion of conditions $q=(D_q, F_q)$ where $D_q \in [\omega_2]^{\leq\omega}$ and $F_q \subseteq Fn_{\leq\omega}(\omega,D_q)$ are such that $|F_q| \leq \omega$ and $[\omega \times D_q]^1 \subseteq F_q$, ordered by $p \leq q$ if $D_p \supseteq D_q$, $F_p \supseteq F_q$ and 
\begin{enumerate}[\emph{(Q)}]
\item given $f \in F_p$, there is $g \in F_q$ such that $G_f \cap (\omega \times D_q) \subseteq G_g$.
\end{enumerate}
\end{definition}

\begin{proposition}\label{sigmaclosed}
$\mathbb{Q}$ is $\sigma$-closed.
\end{proposition}
\begin{proof}
Let $(q_n)_{n \in \omega} \subseteq \mathbb{Q}$ be such that $q_{n+1} \leq q_n$ for all $n \in \omega$, where each $q_n = (D_n, F_n)$. Define $q = (D_q, F_q)$ by $D_q = \bigcup_{n\in \omega} D_n$ and $F_q = \bigcup_{n\in \omega} F_n$. As countable unions of countable sets, $D_q$ and $F_q$ are countable. Also, $[\omega \times D_q]^1 = \bigcup_{n \in \omega} ([\omega \times D_n]^1) \subseteq \bigcup_{n \in \omega} F_n = F_q$ and $F_q \subseteq \bigcup_{n \in \omega} Fn_{\leq \omega}(\omega, D_n) \subseteq Fn_{\leq \omega} (\omega, D_q)$, so that $q \in \mathbb{Q}$.

Fix $n \in \omega$ and let us now prove that $q \leq q_n$. Clearly $D_q \supseteq D_n$ and $F_q \supseteq D_n$. To prove (Q) of Definition \ref{definitionQ}, given $f \in F_q$, there is $k \in \omega$ such that $f \in F_k$. If $k \leq n$, then $q_n \leq q_k$, so that $f \in F_k \subseteq F_n$ and we are done. If $k > n$, since $q_k \leq q_n$, there is $g \in F_n$ such that $G_f \cap (\omega \times D_n) \subseteq G_g$. If suffices now to notice that $g \in F_n \subseteq F_q$.
\end{proof}

\begin{lemma}\label{densityQ}
For every $\xi \in \omega_2$ and every $q\in \Q$ there is
$p\leq q$ such that $\xi\in D_p$.
\end{lemma}
\begin{proof}
Let  $D_q=D_p\cup\{\xi\}$, $F_p=F_q\cup[\omega\times \{\xi\}]^1$.
Clearly $p\in\PP$. As the graphs of  all new partial functions in $F_p$ have empty
intersections with $(\omega \times D_q)$ we conclude that $p\leq q$.
\end{proof}

Given a model $V$ and a $\mathbb{Q}$-generic filter $G$ over $V$, for each $\xi \in \omega_2$ and each $i \in \omega$, we define in $V[G]$ the following set:
$$A_{\xi,i} = \{f \in Fn_{\leq \omega}(\omega, \omega_2): \exists q \in G \text{ such that } f\in F_q \text{ and } f(i)=\xi\}.$$
Let $\mathcal{B}$ be the Boolean algebra generated by $\{A_{\xi, i}: (i,\xi) \in \omega \times \omega_2\}$.

\begin{proposition}
In $V[G]$, we have that for every $i \in \omega$, $(A_{\xi,i})_{\xi<\omega_2}$ are pairwise disjoint nonempty sets such that whenever 
$(i_1,\xi_1),...,(i_k,\xi_k), (i, \xi)$ are distinct
pairs, then 
$$A_{i,\xi}\setminus (A_{i_1,\xi_1}\cup...\cup A_{i_k,\xi_k})\not=\emptyset.$$
Therefore, the Boolean algebra $\mathcal{B}$ generated by them is a c-algebra of cardinality $\omega_2$ and its Stone space $K$ is a uniform Eberlein compact space of weight $\omega_2$.
\end{proposition}
\begin{proof}
By \ref{densityQ}, for all distinct $(i_1,\xi_1),...,(i_k,\xi_k), (i, \xi)$ we have 
$$\{(i,\xi)\}\in A_{i,\xi}\setminus (A_{i_1,\xi_1}\cup...\cup A_{i_k,\xi_k}).$$
So each  $A_{\xi,i}$ is nonempty and the condition (c) of Definition \ref{c-algebra} is satisfied. It also follows directly from the fact that $f$'s are functions and that $(i, \xi) \in A_{\xi, i}$ that $(A_{\xi,i})_{\xi<\omega_2}$ are pairwise disjoint.
\end{proof}

For each $(i, \xi) \in \omega \times \omega_2$, let $\dot{A}_{\xi,i}$ be a $\mathbb{Q}$-name for $A_{\xi,i}$. 

\begin{definition}\label{isomorphicQ}
Given $q_1 = (D_1, F_1), q_2 = (D_2, F_2) \in \mathbb{Q}$, we say that they are isomorphic if there is an order-preserving bijection $e: D_1 \rightarrow D_2$ such that $e|_{D_1 \cap D_2} = id$ and for all $f \in Fn_{\leq\omega}(\omega, \omega_2)$, $f \in F_1$ if and only if $e[f] \in F_2$, where $e[f](i) = e(f(i))$. 
\end{definition}

\begin{lemma}\label{minimalAmalgamationQ}
Let $q_k= (D_k, F_k)$ in $\mathbb{Q}$, for $1 \leq k \leq m$, be pairwise isomorphic conditions such that $(D_k)_{1 \leq k\leq m}$ is a $\Delta$-system with a countable root $D$. 
Then, there is $q \in\mathbb{Q}$, $q \leq q_1, \dots, q_m$
such that for any distinct $i,i'\in \omega$, any distinct $1\leq k,k'\leq m$ and any $\xi\in D_k\setminus D$ and
$\xi'\in D_{k'}\setminus D$ we have 
$$q \Vdash \dot{A}_{\check{\xi},i} \cap \dot{A}_{\check{\xi}',i'}=  \emptyset.$$
\end{lemma}
\begin{proof}
We define $q= (D_q, F_q)$ by letting $D_q = D_1 \cup \dots \cup D_m$ and $F_q = F_1 \cup \dots \cup F_m$. Let us prove that $q \in \mathbb{Q}$ and that $q \leq q_k$ for every $1 \leq k \leq m$. 

As finite unions of countable sets, $D_q$ and $F_q$ are countable. Notice that $\omega \times D_q = (\omega \times D_1) \cup \dots \cup (\omega \times D_m) \subseteq F_1 \cup \dots \cup F_m =F_q$ and that $F_q \subseteq Fn_{\leq \omega}(\omega, D_1) \cup \dots \cup Fn_{\leq \omega} (\omega, D_m) \subseteq Fn_{\leq \omega} (\omega, D_q)$, so that $q \in \mathbb{Q}$.

Fix $1 \leq k \leq m$ and let us now show that $q \leq q_k$. By the definition of $q$, $D_q \supseteq D_k$ and $F_q \supseteq F_k$. Given $f \in F_q$, (Q) of Definition \ref{definitionQ} is trivially satisfied if $f \in F_k$. So, suppose $f \in F_{k'} \setminus F_k$ for some $k'\neq k$ and let $e$ be the order-preserving bijection from $D_{k'}$ onto $D_k$. Since $q_k$ and $q_{k'}$ are isomorphic and $f \in F_{k'}$, we get that $g = e[f] \in F_k$. Let us prove that $G_f \cap (\omega \times D_k) = G_g \cap (\omega \times D) \subseteq G_g$: given $(i,\xi) \in G_f \cap (\omega \times D_k)$, we have that $f(i)=\xi$ and since $f \in F_{k'} \subseteq \omega \times D_{k'}$, we get that $\xi \in D_k \cap D_{k'} = D$. Then $g(i) = e[f](i) = e(f(i)) = e(\xi) = \xi$, since $\xi \in D$. This proves that $G_f \cap (\omega \times D_k) \subseteq G_g$ and concludes the proof.

Now let $i, i'\in \omega$ be distinct, $1\leq k<k'\leq m$ and let $\xi\in D_k\setminus D$ and $\xi'\in D_{k'}\setminus D$. 
Note that in $F_q$ there is no function $f$ such that $f(i)=\xi$ and $f(i')=\xi'$. Let $G$ be a $\mathbb{Q}$-generic filter such that $q\in G$. Given any $q'\in G$, there is $q''\in G$ such that $q''\leq q,q'$. By Definition \ref{definitionQ} (Q), 
no function $g\in F_{q''}$ satisfies $g(i)=\xi$ and $g(i')=\xi'$, hence there is no such function in $F_{q'}$. It follows from the definition 
of $A_{\xi,i}$'s that $A_{\xi,i}\cap A_{\xi',i'}=\emptyset$ in $V[G]$. Since
$G$ was an arbitrary $\mathbb{Q}$-generic filter containing $q$, we conclude that 
$q \Vdash \dot{A}_{\check{\xi},i} \cap \dot{A}_{\check{\xi}',i'}=  \emptyset$.
\end{proof}

\begin{theorem}
Assuming CH, given $(q_\alpha)_{\alpha < \omega_2} \subseteq \mathbb{Q}$, there is $A \in [\omega_2]^{\omega_2}$ such that for every $\alpha_1< \dots < \alpha_m \in A$, there is $q \in \mathbb{Q}$, $q \leq q_{\alpha_1}, \dots q_{\alpha_m}$. In particular, $\mathbb{Q}$ satisfies the  $\omega_2$-c.c.
\end{theorem}
\begin{proof}
Let $(q_\alpha)_{\alpha<\omega_2} \subseteq \mathbb{Q}$, where each $q_\alpha = (D_\alpha, F_\alpha)$. Using CH, by the $\Delta$-system lemma for countable sets (see \cite{kunen}), there is $A'\in [\omega_2]^{\omega_2}$ such that $(D_\alpha)_{\alpha \in A}$ is a $\Delta$-system with countable root $D$. By thinning out using CH and counting arguments, in particular the fact that there are $\mathfrak{c}=\omega_1$ countable sets of the set of $D^\omega$ of cardinality $\mathfrak{c}=\omega_1$, there is $A \in [A']^{\omega_2}$ such that for every $\alpha, \beta \in A$, $q_\alpha$ and $q_\beta$ are isomorphic. Now, given $\alpha_1 < \dots < \alpha_m \in A$, by Lemma \ref{minimalAmalgamation} there is $q \in \mathbb{Q}$, $q \leq q_{\alpha_1}, \dots, q_{\alpha_m}$, which concludes the proof.
\end{proof}

\begin{lemma}\label{strongAmalgamationQ}
Let $q_k= (D_k, F_k)$ in $\mathbb{Q}$, for $1 \leq k \leq m$, be pairwise isomorphic conditions such that $(D_k)_{1 \leq k\leq m}$ is a $\Delta$-system with a countable root $D$. Given, for all $1 \leq k \leq m$, $\xi_k \in D_k \setminus D$ and distinct $i_k \in \omega$, there is $q \in\mathbb{Q}$, $q \leq q_1, \dots, q_m$ such that
$$q \Vdash \dot{A}_{\check{\xi}_1,i_1} \cap \dots \cap \dot{A}_{\check{\xi}_m,i_m} \neq \emptyset.$$
\end{lemma}

\begin{proof}
We define $q= (D_q, F_q)$ by letting  $D_q = D_1 \cup \dots \cup D_m$ and 
$$F_q = F_1 \cup \dots \cup F_m \cup \{f_0\},$$
where $f_0= \{(i_1,\xi_1), \dots, (i_m, \xi_m)\}$.
Let us check that $q \in \mathbb{Q}$ and that $q \leq q_k$ for every $1 \leq k \leq m$. 

As finite unions of countable sets, $D_q$ and $F_q$ are countable. Notice that $[\omega \times D_q]^1 = [\omega \times D_1]^1 \cup \dots \cup [\omega \times D_m]^1 \subseteq F_1 \cup \dots \cup F_m \subseteq F_q$ and that $F_q \setminus \{f_0\} \subseteq Fn_{\leq \omega}(\omega, D_1) \cup \dots \cup Fn_{\leq \omega} (\omega, D_m) \subseteq Fn_{\leq \omega} (\omega, D_q)$. Since $f_0 \in Fn_{\leq \omega}(\omega, D_1 \cup \dots \cup D_m) \subseteq Fn_{\leq \omega} (\omega, D_q)$, we get that $q \in \mathbb{Q}$.

Fix $1 \leq k \leq m$ and let us now verify that $q \leq q_k$. By the definition of $q$, $D_q \supseteq D_k$ and $F_q \supseteq F_k$. To check (Q) of Definition \ref{definitionQ}, given $f \in F_q$, let us consider three cases:

\textbf{Case 1.} If $f \in F_k$, then (Q) is trivially satisfied. 

\textbf{Case 2.} If $f \in F_{k'} \setminus F_k$ for some $k'\neq k$, let $e$ be the order-preserving bijection from $D_{k'}$ onto $D_k$. Since $q_k$ and $q_{k'}$ are isomorphic and $f \in F_{k'}$, we get that $g = e[f] \in F_k$. Let us prove that $G_f \cap (\omega \times D_k) = G_g \cap (\omega \times D) \subseteq G_g$: given $(i,\xi) \in G_f \cap (\omega \times D_k)$, we have that $f(i)=\xi$ and since $f \in F_{k'} \subseteq \omega \times D_{k'}$, we get that $\xi \in D_k \cap D_{k'} = D$. Then $g(i) = e[f](i) = e(f(i)) = e(\xi) = \xi$, since $\xi \in D$. This proves that $G_f \cap (\omega \times D_k) \subseteq G_g$, as we wanted.

\textbf{Case 3.} If $f = f_0= \{(i_1,\xi_1), \dots, (i_m, \xi_m)\}$, then $G_f \cap (\omega \times D_k) = \{(i_k, \xi_k)\} \subseteq F_k$ by Definition \ref{definitionQ}.

Finally, since $f_0 \in F_q$, $q$ forces that $\check{f}_0 \in \dot{A}_{\check{\xi}_1,i_1} \cap \dots \cap \dot{A}_{\check{\xi}_m,i_m}$ and we get that 
$$q \Vdash \dot{A}_{\check{\xi}_1,i_1} \cap \dots \cap \dot{A}_{\check{\xi}_m,i_m} \neq \emptyset,$$
which concludes the proof.
\end{proof}

\begin{definition}
$\Pi$ is the product of $\omega_3$ copies of $\mathbb{Q}$, with countable supports, that is:
$$\Pi = \{\pi: dom(\pi) \rightarrow \mathbb{Q}: dom(\pi) \in [\omega_3]^{\leq \omega}\},$$
ordered by $\pi_1 \leq \pi_2$ if $dom(\pi_1) \supseteq dom(\pi_2)$ and for every $\alpha \in dom(\pi_2)$, $\pi_1(\alpha) \leq_{\Q} \pi_2(\alpha)$. Given $A \subseteq \omega_3$, let 
$$\Pi_A = \{\pi \in \Pi: dom(\pi) \subseteq A\}.$$
$\mathbb{R}$ is the product of Cohen forcing adding $\omega_2$ Cohen reals with $\Pi$, that is,
$$\mathbb{R} = Fn_{<\omega}(\omega_2, 2) \times \Pi.$$
\end{definition}

\begin{proposition}\label{isomorphism}
 Given $\gamma_0 \in \omega_3$, if $A = \omega_3 \setminus \{\gamma_0\}$, then the forcing $\Pi_A$ is isomorphic to $\Pi$ and, therefore, 
 $$\mathbb{R} = Fn_{<\omega}(\omega_2, 2) \times \Pi \sim Fn_{<\omega}(\omega_2, 2) \times \Pi_A \times \mathbb{Q} \sim \mathbb{R} \times \mathbb{Q}.$$
\end{proposition}
\begin{proof}
Lift a bijection between  $\omega_3$ and $\omega_3\setminus\{\gamma_0\}$ to
an isomorphism of $\Pi$ and $\Pi_A$.
\end{proof}

Given a model $V$ of CH and an $\mathbb{R}$-generic filter $G$ over $V$, for each $\gamma_0 \in \omega_3$, each $\xi \in \omega_2$ and each $i \in \omega$, we define in $V[G]$ the following set:
$$A_{\xi,i}(\gamma_0) = \{f \in Fn_{\leq \omega}(\omega, \omega_2)\cap V: \exists r = (C,\pi) \in G \text{ such that } f\in F_{\pi(\gamma_0)} \text{ and } f(i)=\xi\}.$$
Let $\mathcal{B}_{\gamma_0}$ be 
a subalgebra of the Boolean algebra
$\wp(Fn_{\leq \omega}(\omega, \omega_2))$ generated by $\{A_{\xi, i}(\gamma_0): (i,\xi) \in \omega \times \omega_2\}$.

\begin{proposition}
In $V[G]$, we have that for every $\gamma_0 \in \omega_3$ and every $i \in \omega$, $(A_{\xi,i}(\gamma_0))_{\xi<\omega_2}$ are pairwise disjoint nonempty sets such that whenever 
$(i_1,\xi_1),...,(i_k,\xi_k), (i, \xi)$ are distinct
pairs, then 
$$A_{i,\xi}(\gamma_0)\setminus (A_{i_1,\xi_1}(\gamma_0)\cup...\cup A_{i_k,\xi_k}(\gamma_0))\not=\emptyset.$$
Therefore, the Boolean algebra $\mathcal{B}_{\gamma_0}$ generated by them is a c-algebra of cardinality $\omega_2$ and its Stone space $K_{\gamma_0}$ is a uniform Eberlein compact space of weight $\omega_2$.
\end{proposition}
\begin{proof} Fix $\gamma_0\in\omega_3$.
By \ref{densityQ} for all distinct $(i_1,\xi_1),...,(i_k,\xi_k), (i, \xi)$ we have 
$$\{(i,\xi)\}\in A_{i,\xi}(\gamma_0)\setminus (A_{i_1,\xi_1}(\gamma_0)\cup...\cup A_{i_k,\xi_k}(\gamma_0)).$$
So each  $A_{\xi,i}(\gamma_0)$ is nonempty and the condition (c) of Definition \ref{c-algebra} is satisfied. 
It follows directly from the fact that $f$'s are functions and that $(i, \xi) \in A_{\xi, i}(\gamma_0)$ that $(A_{\xi,i}(\gamma_0))_{\xi<\omega_2}$ are pairwise disjoint.
\end{proof}

Since now we work in $V$.
For each $\gamma_0 \in \omega_3$ and each $(i, \xi) \in \omega \times \omega_2$, let $\dot{A}_{\xi,i}(\gamma_0)$ be an $\mathbb{R}$-name for $A_{\xi,i}(\gamma_0)$. 

\begin{lemma}\label{omega2cc}
Assuming CH, given $(r_\alpha)_{\alpha < \omega_2} \subseteq \mathbb{R}$, there is $A \in [\omega_2]^{\omega_2}$ such that for every $\alpha_1< \dots < \alpha_m \in A$, there is $r \in \mathbb{R}$, $r \leq r_{\alpha_1}, \dots r_{\alpha_m}$. In particular, $\mathbb{R}$ is $\omega_2$-cc.
\end{lemma}
\begin{proof}
Let $(r_\alpha)_{\alpha<\omega_2} \subseteq \mathbb{R}$, where each $r_\alpha = (C_\alpha, \pi_\alpha)$. Let $B_\alpha = dom(\pi_\alpha) \in [\omega_3]^{\leq \omega}$. Using CH and the $\Delta$-system lemma for countable sets, there is $A'\in [\omega_2]^{\omega_2}$ such that $(B_\alpha)_{\alpha \in A'}$ is a $\Delta$-system with countable root $B$ and that $(dom(C_\alpha))_{\alpha \in A'}$ is a $\Delta$-system  with finite root $\Delta$.  Using another standard counting argument we may assume that $C_\alpha|_\Delta = C_\beta|_\Delta$ for any $\alpha<\beta \in A'$. By thinning out using further counting arguments
and CH, there is $A \in [A']^{\omega_2}$ such that for every $\xi \in B$, $(D_{r_\alpha(\xi)}))_{\alpha \in A}$ is a $\Delta$-system with countable root $D_\xi$ and $(\pi_\alpha(\xi))_{\alpha \in A}$ are pairwise isomorphic. 

Fix $\alpha_1 < \dots < \alpha_m \in A$. Let us define $r=(C_r, \pi_r)$: put $C_r = C_{\alpha_1} \cup \dots \cup C_{\alpha_m}$ and $dom(\pi_r) = B_{\alpha_1} \cup \dots \cup B_{\alpha_m}$. For each $\xi \in B$, let $\pi_r(\xi) \in \mathbb{Q}$ be the condition given by Lemma \ref{minimalAmalgamationQ}, such that $\pi_r(\xi) \leq \pi_{\alpha_1}(\xi), \dots, \pi_{\alpha_m}(\xi)$. If $\xi \in B_{\alpha_k} \setminus B$, let $\pi_r(\xi) = \pi_{\alpha_k}(\xi)$.

Clearly $r \in \mathbb{R}$ and $C_r \leq C_{\alpha_k}$ by the definition, for all $1 \leq k \leq m$. Also, both for $\xi \in B$ and for $\xi \in B_{\alpha_k} \setminus B$, $\pi_r(\xi) \leq \pi_{\alpha_k}(\xi)$, so that $r \leq r_{\alpha_k}$ for all $1 \leq k \leq m$.
\end{proof}

\begin{theorem}
If $CH$ holds in $V$, then $V^{\mathbb{R}} \vDash$ ``there is no Banach space $X$ of density $\mathfrak{c} = \omega_2$ such that for every uniform Eberlein compact space $K$ of weight at most $\mathfrak{c}$, $C(K)$ can be isomorphically embedded into $X$''.
\end{theorem}
\begin{proof}
First note that $\R \Vdash \mathfrak{c} = \omega_2$. By the product lemma (\cite{kunen}),
for any generic $G\subseteq \R$, the extension $V[G]$ is of the form
$V[H_1][H_2]$ where $H_1$ is a $\Pi$-generic over $V$
and $H_2$ is $\check Fn_{<\omega}(\omega_2, 2)$-generic over $V[H_1]$. 
The first forcing $\Pi$ satisfies the $\omega_2$-c.c. by Lemma \ref{omega2cc} and is $\sigma$-closed
by Proposition \ref{sigmaclosed}, so $V[H_1]$ has the same cardinals as in $V$ and satisfies the CH.
The second forcing is equal to $Fn_{<\omega}(\omega_2, 2)$ since finite functions are
the same in $V$ and $V[H_1]$. Hence $V[G]$ can be viewed as an extension of a model
of CH by the Cohen forcing which adds $\omega_2$ Cohen reals. It is well-known
(see \cite{kunen}) that  $\mathfrak{c} = \omega_2$ holds in such models.

We work in $V$. By contradiction, suppose that in $V^\mathbb{R}$ 
there is a Banach space which contains an isomorph of every UG Banach space of
density $\omega_2$. Let $\dot{X}$ be an $\mathbb{R}$-name for it. Since $\mathbb{R} \Vdash |d(\dot{X})| = \mathfrak{c} = \omega_2$, we get that $\mathbb{R} \Vdash |\dot{X}| = \mathfrak{c}^\omega = \omega_2$, so that there is a family of $\mathbb{R}$-names $(\dot{v}_\eta)_{\eta <\omega_2}$ such that $\mathbb{R} \Vdash \dot{X} = \{\dot{v}_\eta: \eta < \omega_2\}$. 

For each $F \in [\omega_2]^{<\omega}$, let $A_F \subseteq \mathbb{R}$ be a maximal antichain in $\mathbb{R}$ such that for every $r \in A_F$, 
$$\text{either }r \Vdash \Vert \sum_{\eta\in F}\dot{v}_{\eta} \Vert > 1 \quad \text{ or }r \Vdash \Vert \sum_{\eta\in F}\dot{v}_{\eta} \Vert \leq 1.$$
Since $\mathbb{R}$ is $\omega_2$-cc, each $A_F$ has cardinality at most $\omega_1$ and for each $(C,\pi) \in A_F$, $|dom(\pi)| \leq \omega$. Then, the set
$$\bigcup \{dom(\pi): (C,\pi) \in A_F \text{ for some } F \in [\omega_2]^{<\omega} \text{ and } C \in Fn_{<\omega}(\omega_2,2) \}$$
has cardinality at most $\omega_2$. Fix $\gamma_0 \in \omega_3$ such that for every $F \in [\omega_2]^{<\omega}$ and every $(C, \pi) \in A_F$, $\gamma_0 \notin dom(\pi)$.

In $V^\mathbb{R}$, let $K_{\gamma_0}$ be the Stone space of the c-algebra generated by the family $(A_{\xi,i}(\gamma_0))_{\xi\in \omega_2, i \in \omega}$ and let us show that $C(K_{\gamma_0})$ does not embed isomorphically into $\dot{X}$. For each $(i, \xi) \in \omega \times \omega_2$, let $\dot{A}_{\xi, i}$ be an $\mathbb{R}$-name for $A_{\xi,i}(\gamma_0)$.

Suppose there is $\dot{T}: C(K_{\gamma_0}) \rightarrow \dot{X}$ an isomorphic embedding and without loss of generality,  assume that $\mathbb{R} \Vdash \Vert \dot{T} \Vert = 1$.
Let $r'\in \R$ and find
$r''\leq r'$ and $m \in \omega$ such that $r'' \Vdash \Vert \dot{T}^{-1} \Vert <\check m$.

For each $\alpha< \omega_2$, let $\{\xi_k(\alpha): 1 \leq k \leq m\} \subseteq \omega_2 \setminus \alpha$ 
have cardinality $m$ and take $r_\alpha \in \mathbb{R}$ 
with $r_\alpha\leq r''$ and $\eta_1(\alpha), \dots \eta_m(\alpha) \in \omega_2$ such that 
$$r_\alpha \Vdash \forall 1 \leq k \leq m \quad \dot{T}(\chi_{[\dot{A}_{\check{\xi}_k(\alpha), k}(\gamma_0)]}) = \dot{v}_{\check{\eta}_k(\alpha)}.$$
Let $r_\alpha = (C_\alpha, \pi_\alpha)$, $s_\alpha = (C_\alpha, \pi_\alpha|_{\omega_2 \setminus \{\gamma_0\}}) \in Fn(\omega_2,2) \times \Pi_{\omega_2 \setminus \{\gamma_0\}}$ and $\pi_\alpha(\gamma_0) = (D_\alpha, F_\alpha)$. Without loss of generality, by \ref{densityQ} we may assume that $\xi_1(\alpha), \dots , \xi_m(\alpha) \in D_\alpha$. Using CH and the $\Delta$-system lemma for countable sets, we may assume that $(D_\alpha)_{\alpha<\omega_2}$ is a $\Delta$-system with a countable root $D$ and that $\xi_1(\alpha), \dots , \xi_m(\alpha) \in D_\alpha \setminus D$ for each $\alpha<\omega_2$
and that for each $\alpha<\omega_2$
the conditions $\pi_\alpha(\gamma_0)$ are isomorphic in the sense of Definition \ref{isomorphicQ}.

By Proposition \ref{isomorphism}, $Fn(\omega_2,2) \times \Pi_{\omega_2 \setminus \{\gamma_0\}}$ is isomorphic to $\mathbb{R}$. So, given $\alpha_1< \dots < \alpha_m$, we can apply Lemma \ref{omega2cc} to $s_{\alpha_1}, \dots, s_{\alpha_m}$ and get $s \in Fn(\omega_2,2) \times \Pi_{\omega_2 \setminus \{\gamma_0\}}$ such that $s \leq s_{\alpha_1}, \dots, s_{\alpha_m}$.

Let $F = \{ \eta_1(\alpha_1), \dots, \eta_m(\alpha_m) \}$ and since $A_F$ is a maximal antichain in $\mathbb{R}$ which is contained in $Fn(\omega_2,2) \times \Pi_{\omega_2 \setminus \{\gamma_0\}}$, $A_F$ is also a maximal antichain in $Fn(\omega_2,2) \times \Pi_{\omega_2 \setminus \{\gamma_0\}}$. Then, there is $s' \in A_F$ and $s'' \in Fn(\omega_2,2) \times \Pi_{\omega_2 \setminus \{\gamma_0\}}$ such that $s'' \leq s, s'$. 
Since $s''\leq s'$ and $s'\in A_F$, either
$$s'' \Vdash \Vert \dot{v}_{\check{\eta}_1(\alpha_1)} + \dots + \dot{v}_{\check{\eta}_m(\alpha_m)} \Vert \leq 1$$
or 
$$s'' \Vdash \Vert \dot{v}_{\check{\eta}_1(\alpha_1)} + \dots + \dot{v}_{\check{\eta}_m(\alpha_m)} \Vert > 1.$$

\textbf{Case 1.} $s''$ forces that $\Vert \dot{v}_{\check{\eta}_1(\alpha_1)} + \dots + \dot{v}_{\check{\eta}_m(\alpha_m)} \Vert > 1$.

In this case, let $q \in \mathbb{Q}$ be such that $q \leq \pi_{\alpha_1}(\gamma_0), \dots, \pi_{\alpha_m}(\gamma_0)$ obtained by Lemma \ref{minimalAmalgamationQ} and let $r = (s'', q) \in \mathbb{R}$. Then, $r \leq s''$, $r \leq r_{\alpha_1}, \dots, r_{\alpha_m}$ and $r \leq q\leq \pi_{\alpha_1}(\gamma_0), \dots, \pi_{\alpha_m}(\gamma_0)$. So,
$$r \Vdash \forall 1 \leq k < k'\leq m \quad \dot{A}_{\check{\xi}_k(\alpha_k),i_k}(\gamma_0)\cap \dot{A}_{\check{\xi}_{k'}(\alpha_{k'}),i_{k'}}(\gamma_0) = \emptyset$$
so that 
$$r \Vdash \Vert \chi_{[\dot{A}_{\check{\xi}_1(\alpha_1),i_1}(\gamma_0)]} + \dots + \chi_{[\dot{A}_{\check{\xi}_m(\alpha_m),i_m}(\gamma_0)]} \Vert = 1.$$
Since $r \leq r_{\alpha_1}, \dots, r_{\alpha_m}$, 
$$r \Vdash \forall 1 \leq k \leq m \quad \dot{T}(\chi_{[\dot{A}_{\check{\xi}_k(\alpha_k),i_k}(\gamma_0)]}) = \dot{v}_{\check{\eta}_k(\alpha_k)},$$
so that, using the fact that $\mathbb{R} \Vdash \Vert \dot{T} \Vert =1$, we get that
$$r \Vdash \Vert \dot{v}_{\check{\eta}_1(\alpha_1)} + \dots + \dot{v}_{\check{\eta}_m(\alpha_m)} \Vert \leq 
\Vert \chi_{[\dot{A}_{\check{\xi}_1(\alpha_1),i_1}(\gamma_0)]} + \dots + \chi_{[\dot{A}_{\check{\xi}_m(\alpha_m),i_m}(\gamma_0)]} \Vert = 1,$$
contradicting the fact that $r \leq s''$ and 
$$s''\Vdash \Vert \dot{v}_{\check{\eta}_1(\alpha_1)} + \dots + \dot{v}_{\check{\eta}_m(\alpha_m)} \Vert > 1.$$

\textbf{Case 2.} $s''$ forces that $\Vert \dot{v}_{\check{\eta}_1(\alpha_1)} + \dots + \dot{v}_{\check{\eta}_m(\alpha_m)} \Vert \leq 1$.

In this case, let $q \in \mathbb{Q}$ be such that $q \leq \pi_{\alpha_1}(\gamma_0), \dots, \pi_{\alpha_m}(\gamma_0)$ obtained by Lemma \ref{strongAmalgamationQ} and let $r = (s'', q) \in \mathbb{R}$. Then, $r \leq s''$, $r \leq r_{\alpha_1}, \dots, r_{\alpha_m}$ and $r \leq q\leq \pi_{\alpha_1}(\gamma_0), \dots, \pi_{\alpha_m}(\gamma_0)$. So,
$$r \Vdash \dot{A}_{\check{\xi}_1(\alpha_1),i_1}(\gamma_0) \cap \dots \cap \dot{A}_{\check{\xi}_{m}(\alpha_{m}),i_{m}}(\gamma_0) \neq \emptyset$$
so that 
$$r \Vdash \Vert \chi_{[\dot{A}_{\check{\xi}_1(\alpha_1),i_1}(\gamma_0)]} + \dots + \chi_{[\dot{A}_{\check{\xi}_m(\alpha_m),i_m}(\gamma_0)]} \Vert = m.$$
Since $r \leq r_{\alpha_1}, \dots, r_{\alpha_m}$, 
$$r \Vdash \forall 1 \leq k \leq m \quad \dot{T}(\chi_{[\dot{A}_{\check{\xi}_k(\alpha_k),i_k}(\gamma_0)]}) = \dot{v}_{\check{\eta}_k(\alpha_k)},$$
so that, using the fact that $\mathbb{R} \Vdash \Vert \dot{T}^{-1} \Vert < m$ and that $r \leq s''$, we get that
$$r \Vdash \Vert \chi_{[\dot{A}_{\check{\xi}_1(\alpha_1),i_1}]} + \dots + \chi_{[\dot{A}_{\check{\xi}_m(\alpha_m),i_m}]} \Vert \leq
\Vert \dot{T}^{-1} \Vert \cdot \Vert \dot{v}_{\check{\eta}_1(\alpha_1)} + \dots + \dot{v}_{\check{\eta}_m(\alpha_m)}\Vert 
< m,$$
which contradicts our assumption.

Since the condition $r'\in\R$ was arbitrary, we showed that a dense subset of $\R$ forces
the nonexistence of the embedding $\dot T$, hence it does not exist in $V[G]$.
\end{proof}

\bibliographystyle{amsplain}

\end{document}